\documentclass[11pt]{amsart}
\usepackage{amssymb}
\newtheorem{define}{Definition}
\newtheorem{propo}{ Proposition}

\newtheorem{coro}{Corollary}
\newtheorem{theorem}{Theorem}

\begin{document}

\title [Paley-Wiener Theorem for  the Weinstein Transform ]{Paley-Wiener Theorem for  the Weinstein Transform and applications\\}%

\author[ Khaled Mehrez]{KHALED MEHREZ }
\address{Khaled Mehrez. D\'epartement de Math\'ematiques ISSAT Kasserine, Universit\'e de Kairouan, Tunisia.}
 \email{k.mehrez@yahoo.fr}
\begin{abstract} In this paper our aim is to establish the Paley-Wiener Theorems for  the Weinstein Transform. Furthermore, some applications are presents, in particular some properties for the generalized translation operator associated with the  Weinstein operator are proved.
\end{abstract}
\maketitle
\noindent\textbf{keywords:} Weinstein transform, Paley-Wiener Theorem.
\section{Introduction}

 Paley-Wiener theorem is any theorem that relates decay properties of a function or distribution at infinity with analyticity of its Fourier transform \cite{PW}.  For example, the classical Paley-Wiener theorem (see \cite{++}) stated that  states  that  a  necessary  and  sufficient
condition  for  a  square-integrable  function $\varphi$ to  be  extendable to an entire  function  in the  complex  plane  with  an  exponential  type  bound 
$$\left|\varphi(z)\right|\leq C e^{A z}$$
if $\varphi$ is band-limited, i.e. the Fourier transform $\hat{\varphi}$ has compact support.  Higher dimensional extensions of the Paley-Wiener theorem have been studied. There are plenty of Paley-Wiener type theorems since there are many kinds of bound for
decay rates of functions and many types of characterizations of smoothness. In this regard a wide number of papers have been devoted to the extension of the theory on many other transforms and different classes of functions, for example, the Mellin transform \cite{M}, Hankel transform \cite{Kr, V}, Jacobi transform \cite{V1} and Clifford-Fourier transform \cite{SV,VL} .

Since  the  Weinstein transform  are  natural generalizations of  the  Fourier transforms,  it is natural  to ask  whether such  a  representation  for  entire functions  is possible in  this case also.  The  aim  of  this paper  is  to  obtain  an  analogue  of  the  Paley-Wiener  theorem  for Weinstein  transforms. As applications  of this results, some properties for the generalized translation operator associated with the  Weinstein operator are  established.
\\
\section{Harmonic analysis Associated with the Weinstein Operator}

In order to set up basic and standard notation we briefly overview the Weinstein
operator and related harmonic analysis. Main references are \cite{B1, B2}.

In the following we denote by 

\begin{itemize}
\item $\mathbb{R}^{d+1}_+=\mathbb{R}^d\times (0,\infty)$.
\item $x=\left(x_1,...,x_d,x_{d+1}\right)=(x^\prime,x_{d+1}).$
\item $-x=(-x^\prime,x_{d+1}).$
\item $C_*(\mathbb{R}^{d+1})$, the space of continuous functions on $\mathbb{R}^{d+1},$ even with respect to the last variable.
\item $S_*(\mathbb{R}^{d+1})$, the space of the $C^\infty$ functions, even with respect to the last variable, and rapidly decreasing together with their derivatives.
	\item $L^p_\alpha(\mathbb{R}^{d+1}_+),\;1\leq p\leq \infty,$ the space of measurable functions $f$ on $\mathbb{R}^{d+1}_+$ such that 
	$$\left\|f\right\|_{\alpha,p}=\left(\int_{\mathbb{R}^{d+1}_+}\left|f(x)\right|^pd\mu_\alpha(x)\right)^{1/p}<\infty, \;p\in[1,\infty),$$
	$$\left\|f\right\|_{\alpha,\infty}=\textrm{ess}\sup_{x\in\mathbb{R}^{d+1}_+}\left|f(x)\right|<\infty,$$
	where 
	\begin{equation}\label{mesure}
	d\mu_\alpha(x)=x^{2\alpha+1}_{d+1}dx=x^{2\alpha+1}_{d+1}dx_1...dx_{d+1}.
	\end{equation}
	\item $\mathcal{P}_{*,l}^{d+1}$ the set of homogeneous polynomials on $\mathbb{R}^{d+1}$ of degree $l$, even with respect to the last variable. 
	\item $\mathcal{A}_\alpha(\mathbb{R}^{d+1})=\left\{\varphi\in L^1_\alpha(\mathbb{R}^{d+1}_+);\;\mathcal{F}_{W,\alpha}\varphi\in L^1_\alpha(\mathbb{R}^{d+1}_+)\right\}$ the Wiener algebra space. 
	\item $S^d_+=\left\{x\in\mathbb{R}^{d+1}_+;\;\left\|x\right\|=1\right\}.$
\end{itemize}

We consider  the Weinstein operator $\Delta_{W,\alpha}^d$ defined on $\mathbb{R}^{d+1}_+$ by
\begin{equation}
\Delta_{W,\alpha}^d=\sum_{j=1}^{d+1}\frac{\partial^2}{\partial x_j^2}+\frac{2\alpha+1}{x_{d+1}}\frac{\partial}{\partial x_{d+1}}=\Delta_d+L_\alpha,\;\alpha>-1/2,
\end{equation}
where $\Delta_d$ is the Laplacian operator for the $d$ first variables and $L_\alpha$ is the Bessel operator for the last variable defined on $(0,\infty)$ by 
$$L_\alpha u=\frac{\partial^2 u}{\partial x_{d+1}^2}+\frac{2\alpha+1}{x_{d+1}}\frac{\partial u}{\partial x_{d+1}}.$$

The Weinstein operator $\Delta_{W,\alpha}^d$  have remarkable applications in diffrerent branches of
mathematics. For instance, they play a role in Fluid Mechanics \cite{MB}.
\subsection{The eigenfunction of the Weinstein operator }
For all $\lambda=(\lambda_1,...,\lambda_{d+1})\in\mathbb{C}^{d+1}$, the system 
\begin{equation}
\begin{cases}
\frac{\partial^{2}u}{\partial x_{j}^{2}}(x)=-\lambda_{j}^{2}u(x), & \textrm{if}\;1\leq j\leq d
\\L_\alpha u(x)=-\lambda_{d+1}^{2}u(x),
\\ u(0)=1,\;\frac{\partial u}{\partial x_{d+1}}(0)=0,\;\textrm{and}\;\frac{\partial u}{\partial x_{j}}(0)=-i\lambda_j, & \textrm{if}\;1\leq j\leq d,
\end{cases}
\end{equation}
 has a unique solution on $\mathbb{R}^{d+1}$, denoted by $\Lambda_{\alpha}^d(\lambda,.),$ and given by
\begin{equation}
\Lambda_{\alpha}^d(\lambda,x)=e^{-i<x^\prime,\lambda^\prime>}j_\alpha(x_{d+1}\lambda_{d+1})
\end{equation}
 where $x=(x^\prime,x_{d+1}),\; \lambda=(\lambda^\prime,\lambda_{d+1})$ and $j_\alpha$ is is the normalized Bessel function of index $\alpha$ defined by
$$j_\alpha(x)=\Gamma(\alpha+1)\sum_{k=0}^\infty\frac{(-1)^k x^{2k}}{2^k k!\Gamma(\alpha+k+1)}.$$
The function $(\lambda,x)\mapsto\Lambda_{\alpha}^d(\lambda,x)$ has a unique extension to $\mathbb{C}^{d+1}\times\mathbb{C}^{d+1}$, and satisfied the following properties:\\
\begin{propo}
\noindent i). For all $(\lambda,x)\in \mathbb{C}^{d+1}\times\mathbb{C}^{d+1}$ we have 
\begin{equation}
\Lambda_{\alpha}^d(\lambda,x)=\Lambda_{\alpha}^d(x,\lambda)
\end{equation}
\noindent ii). For all $(\lambda,x)\in \mathbb{C}^{d+1}\times\mathbb{C}^{d+1}$ we have 
\begin{equation}
\Lambda_{\alpha}^d(\lambda,-x)=\Lambda_{\alpha}^d(-\lambda,x)
\end{equation}
\noindent iii). For all $(\lambda,x)\in \mathbb{C}^{d+1}\times\mathbb{C}^{d+1}$ we get
\begin{equation}
\Lambda_{\alpha}^d(\lambda,0)=1.
\end{equation}
\noindent vi). For all $\nu\in\mathbb{N}^{d+1},\;x\in\mathbb{R}^{d+1}$ and $\lambda\in\mathbb{C}^{d+1}$ we have
\begin{equation}\label{klk}
 \left|D_\lambda^\nu\Lambda_{\alpha}^d(\lambda,x)\right|\leq\left\|x\right\|^{\left|\nu\right|}e^{\left\|x\right\|\left\|\Im \lambda\right\|}
\end{equation}
where $D_\lambda^\nu=\partial^\nu/(\partial\lambda_1^{\nu_1}...\partial\lambda_{d+1}^{\nu_{d+1}})$ and $\left|\nu\right|=\nu_1+...+\nu_{d+1}.$ In particular, for all $(\lambda,x)\in \mathbb{R}^{d+1}\times\mathbb{R}^{d+1}$, we have
\begin{equation}
\left|\Lambda_{\alpha}^d(\lambda,x)\right|\leq 1.
\end{equation}
\end{propo}
\subsection{The Weinstein transform}
\begin{define} The Weinstein transform  is given for $\varphi\in L^1_\alpha(\mathbb{R}^{d+1}_+)$  by 
\begin{equation}
\mathcal{F}_{W,\alpha}(\varphi)(\lambda)=\int_{\mathbb{R}^{d+1}_{+}}\varphi(x)\Lambda_{\alpha}^d(\lambda,x)d\mu_\alpha(x),\;\;\lambda\in\mathbb{R}^{d+1}_{+},
\end{equation}
where $\mu_\alpha$ is the measure on $\mathbb{R}^{d+1}_+$ given by the relation (\ref{mesure}).
\end{define}

Some basic properties of this transform are as follows. For the proofs, we refer \cite{B1, B2}.  

\begin{propo}
\begin{enumerate}
	\item For all $\varphi\in L^1_\alpha(\mathbb{R}^{d+1}_+)$, the function $\mathcal{F}_{W,\alpha}(\varphi)$  is continuous on $\mathbb{R}^{d+1}_+$ and we have
	\begin{equation}
	\left\|\mathcal{F}_{W,\alpha}\varphi\right\|_{\alpha,\infty}\leq\left\|\varphi\right\|_{\alpha,1}.
	\end{equation}
	\item   The Weinstein transform is a topological isomorphism from $\mathcal{S}_*(\mathbb{R}^{d+1}_+)$ onto itself. The inverse transform is given by
	\begin{equation}
	\mathcal{F}_{W,\alpha}^{-1}\varphi(\lambda)=C_{\alpha,d} \mathcal{F}_{W,\alpha}\varphi(-\lambda),\;\textrm{for\;all}\;\lambda\in\mathbb{R}^{d+1}_+,
	\end{equation}
	where 
	\begin{equation}
	C_{\alpha,d}=\frac{1}{(2\pi)^d2^{2\alpha}\Gamma^2(\alpha+1)}.
	\end{equation}
	\item Parseval formula: For all $\varphi, \phi\in \mathcal{S}_*(\mathbb{R}^{d+1}_+)$, we have
	\begin{equation}\label{MM}
	\int_{\mathbb{R}^{d+1}_+}\varphi(x)\overline{\phi(x)}d\mu_\alpha(x)=C_{\alpha,d}\int_{\mathbb{R}^{d+1}_+}\mathcal{F}_{W,\alpha}(\varphi)(x)\overline{\mathcal{F}_{W,\alpha}(\phi)(x)}d\mu_\alpha(x)
	\end{equation}
\item Plancherel formula: For all $\varphi\in \mathcal{S}_*(\mathbb{R}^{d+1}_+)$, we have
\begin{equation}
\int_{\mathbb{R}^{d+1}_+}\left|\varphi(x)\right|^2d\mu_\alpha(x)=C_{\alpha,d}\int_{\mathbb{R}^{d+1}_+}\left|\mathcal{F}_{W,\alpha}\varphi(x)\right|^2d\mu_\alpha(x)
\end{equation}
\item Inversion formula: If $\varphi\in\mathcal{A}_\alpha(\mathbb{R}^{d+1}_+)$, then 
\begin{equation}\label{inv}
\varphi(\lambda)=C_{\alpha,d}\int_{\mathbb{R}^{d+1}_+}\mathcal{F}_{W,\alpha}\varphi(x)\Lambda_{\alpha}^d(-\lambda,x)d\mu_\alpha(x),\;\textrm{a.e. }\lambda\in\mathbb{R}^{d+1}_+
\end{equation}
\end{enumerate}
\end{propo}
\subsection{The translation operator associated with the Weinstein operator}
\begin{define} The translation operator $\tau^\alpha_x,\;x\in\mathbb{R}^{d+1}_+$ associated with the Weinstein operator $\Delta_\alpha^d$, is defined for a continuous function $\varphi$ on $\mathbb{R}^{d+1}_+$ which is even with respect to the last variable and for all $y\in\mathbb{R}^{d+1}_+$ by
$$\tau^\alpha_x\varphi(y)=\frac{\Gamma(\alpha+1)}{\sqrt{\pi}\Gamma(\alpha+1/2)}\int_0^\pi\varphi\left(x^\prime+y\prime,\sqrt{x^2_{d+1}+y^2_{d+1}+2x_{d+1}y_{d+1}\cos\theta}\right)\left(\sin\theta\right)^{2\alpha}d\theta.$$
\end{define}
By using the Weinstein kernel, we can also define a generalized translation, for
a function $\varphi\in\mathcal{S}_*(\mathbb{R}^{d+1})$ and $y\in\mathbb{R}^{d+1}_+$ the generalized translation $\tau^\alpha_x\varphi$ is defined by the following relation
\begin{equation}\label{MMM}
\mathcal{F}_{W,\alpha}(\tau^\alpha_x\varphi)(y)=\Lambda^d_\alpha(x,y)\mathcal{F}_{W,\alpha}(\varphi)(y)
\end{equation}

The following proposition summarizes some properties of the Weinstein
translation operator.

\begin{propo} The translation operator $\tau^\alpha_x,\;x\in\mathbb{R}^{d+1}_+$ satisfies the following properties:\\ 
i). For $\varphi\in\mathbb{C}_*(\mathbb{R}^{d+1})$, we have for all $x,y\in\mathbb{R}^{d+1}_+$
$$\tau^\alpha_x\varphi(y)=\tau^\alpha_y\varphi(x)\;\textrm{and}\;\tau^\alpha_0\varphi=\varphi.$$
ii). Let $\varphi\in L^p_\alpha(\mathbb{R}^{d+1}_+),\;1\leq p\leq \infty$ and $x\in\mathbb{R}^{d+1}_+$. Then  $\tau^\alpha_x\varphi$ belongs to $L^p_\alpha(\mathbb{R}^{d+1}_+)$ and we have
$$\left\| \tau^\alpha_x\varphi\right\|_{\alpha,p}\leq \left\|\varphi\right\|_{\alpha,p}$$
\end{propo}

Note that the  $\mathcal{A}_\alpha(\mathbb{R}^{d+1}_+)$ is contained in the intersection of $L^1_\alpha(\mathbb{R}^{d+1}_+)$ and $L^\infty_\alpha(\mathbb{R}^{d+1}_+)$  and hence is a subspace of $L^2_\alpha(\mathbb{R}^{d+1}_+).$ For $\varphi\in\mathcal{A}_\alpha(\mathbb{R}^{d+1}_+)$ we have
\begin{equation}\label{tau}
\tau^\alpha_x\varphi(y)=C_{\alpha,d} \int_{\mathbb{R}^{d+1}_+}\Lambda^d_\alpha(x,z)\Lambda^d_\alpha(-y,z)\mathcal{F}_{W,\alpha}\varphi(z)d\mu_{\alpha}(z).
\end{equation}

By using the generalized translation, we define the generalized convolution product $\varphi*_W\psi$ of the functions $\varphi,\;\psi\in L^1_\alpha(\mathbb{R}^{d+1}_+)$ as follows
\begin{equation}
\varphi*_W\psi(x)=\int_{\mathbb{R}^{d+1}_+}\tau^\alpha_x\varphi(-y)\psi(y)d\mu_\alpha(y).
\end{equation}
\\
This convolution is commutative and associative, and it satisfies the following properties:
\\
\begin{propo}
i). For all $\varphi,\psi\in L^1_\alpha(\mathbb{R}^{d+1}_+),$\;(resp. $\varphi,\psi\in \mathcal{S}_*(\mathbb{R}^{d+1}_+)$), then $\varphi*_W\psi\in L^1_\alpha(\mathbb{R}^{d+1}_+),$\;(resp. $\varphi*_W\psi\in \mathcal{S}_*(\mathbb{R}^{d+1}_+)$) and we have
\begin{equation}
\mathcal{F}_{W,\alpha}(\varphi*_W\psi)=\mathcal{F}_{W,\alpha}(\varphi)\mathcal{F}_{W,\alpha}(\psi)
\end{equation}
ii). Let $p, q, r\in [1,\infty],$ such that $\frac{1}{p}+\frac{1}{q}-\frac{1}{r}=1.$ Then for all $\varphi\in L^p_\alpha(\mathbb{R}^{d+1}_+)$ and  $\psi\in L^q_\alpha(\mathbb{R}^{d+1}_+)$ the function $\varphi*_W\psi$ belongs to  $L^r_\alpha(\mathbb{R}^{d+1}_+)$ and we have
\begin{equation}
\left\|\varphi*_W\psi\right\|_{\alpha,r}\leq\left\|\varphi\right\|_{\alpha,p}\left\|\psi\right\|_{\alpha,q}.
\end{equation}
\end{propo}
\subsection{Heat functions related to the Weinstein operator}

The generalized heat kernel $E^\alpha_t(x),\;x\in\mathbb{R}^{d+1}_+,\;t>0$  associated with the Weinstein operator $\Delta^d_{W,\alpha}$ is given by 
\begin{equation}
E^\alpha_t(x)=\frac{2}{\pi^{\frac{d}{2}}\Gamma(\alpha+1)(4t)^{\alpha+1+d/2}}e^{-\left\|y\right\|^2/4t},
\end{equation}
which is a solution of the generalized heat equation:
$$\frac{\partial}{\partial t}E^\alpha_t(x)-\Delta^d_{W,\alpha}E^\alpha_t(x)=0,$$
and satisfied the following properties ( see \cite{Khaled}):\\
\begin{propo}
i). For $x\in\mathbb{R}^{d+1}_+,\;t>0$,  the function $\tau^\alpha_xE^\alpha_t$ is nonnegative.\\
ii). For $t>0$, we have
\begin{equation}\label{hhh}
\left\|E^\alpha_t\right\|_{\alpha,1}=1.
\end{equation}
\noindent iii).  For $x\in\mathbb{R}^{d+1}_+,\;t>0$, we have
\begin{equation}\label{;;} 
\int_{\mathbb{R}^{d+1}_+}\tau^\alpha_x E^\alpha_t(y)d\mu_\alpha(y)=1.
\end{equation}
iv). For $t>0$, we have
\begin{equation}\label{MMMM} 
\mathcal{F}_{W,\alpha}E^\alpha_t(x)=e^{-t\left\|x\right\|^2}
\end{equation}
\end{propo}
\section{Paley-Wiener Theorem for  the Weinstein Transform}

In this section we prove a sharp Paley-Wiener theorem for the Weinstein transform and study its consequences. We suppose that $d\geq1$ and $\lambda>0.$ For a non-negative integer $l$, we put 
$$\mathcal{H}_l^\alpha=\left\{P\in\mathcal{P}_{*,l};\:P\;\textrm{is\;homogeneous\;such\; that }\Delta_\alpha P=0 \right\},$$
which is called the space of generalized spherical harmonics of degree $l.$ We fix a $P_l\in\mathcal{H}_l^\alpha$ and define the Weinstein coefficients of  $\varphi\in\mathcal{S}(\mathbb{R}^{d+1}_+)$  in the angular variable by
\begin{equation}\label{888} 
\varphi_{l,\alpha}(\lambda)=\int_{S^d_+}\varphi(\lambda t)P_l(t)d\sigma_\alpha(t),
\end{equation}
with $d\sigma_\alpha(t)=t^{2\alpha+1}_{d+1}d\sigma_{d+1}(t)$. Then the Weinstein spherical harmonic coefficients of $\varphi\in\mathcal{S}(\mathbb{R}^{d+1}_+)$ are given by 
\begin{equation}\label{889} 
\Phi_{l,\alpha}(\lambda)=\lambda^{-1}\int_{S^d_+}\mathcal{F}_{W,\alpha}(\varphi)(\lambda,t)P_l(t)d\sigma_\alpha(t)
\end{equation}

\begin{theorem}\label{tyt}Let $\varphi\in\mathcal{S}_*(\mathbb{R}^{d+1}_+)$ and $R$ be a positive number. Then $\varphi$  is supported in $\left\{x;\;\left\|x\right\|<R\right\}$ if and only if the  Weinstein spherical harmonic coefficients of $\varphi$ extends to an entire function of $\lambda\in\mathbb{C}$ satisfying the estimate
$$\left|\Phi_{l,\alpha}(\lambda)\right|\leq c_{l,\alpha}e^{R\left\|\Im(\lambda)\right\|}.$$
\end{theorem}
\begin{proof} In \cite{HM} (see proof of Proposition 5), the authors proved that 
$$\Phi_{l,\alpha}(\lambda)=C_{l,\alpha}\int_0^\infty\varphi_{l,\alpha}(r)j_{\alpha+\frac{d}{2}+l}(\lambda r)r^{2\alpha+2l+d+1}dr,$$
where $\varphi_{l,\alpha}(r)$ as defined in (\ref{888}). Thus,  $\Phi_{l,\alpha}$ is the Hankel transform of order $\alpha+\frac{d}{2}+l$ of the function $\varphi_{l,\alpha}(l).$ So,  Paley-Wiener  Theorem for the Hankel transform ( see \cite{Kr}) and (\ref{888}) completes the proof.
\end{proof}

\begin{theorem} A $\varphi\in \mathcal{S}_*(\mathbb{R}^{d+1}_+)$ be supported in $\left\{x;\;\left\|x\right\|<R\right\}$ if and only if the function $\mathcal{F}_{W,\alpha}\varphi$ extends to an entire function of $y\in\mathbb{C}^{d+1}$ which satisfies
	\begin{equation}\label{pw}
	\left|\mathcal{F}_{W,\alpha}\varphi(y)\right|\leq c e^{R\left\|\Im(y)\right\|}.
	\end{equation}
\end{theorem}
\begin{proof} Since the Weinstein kernel $\Lambda^d_\alpha(x,y)$ is an entire function in $y\in\mathbb{C}^{d+1}$  and satisfied 
$$\left|\Lambda^d_\alpha(x,y)\right|\leq  e^{\left\|x\right\|\left\|\Im(y)\right\|},$$
On the other hand,  by the definition of the Weinstein transform we have
$$\left|\mathcal{F}_{W,\alpha}\varphi(y)\right|\leq c e^{R\left\|\Im(y)\right\|},$$
where $c=\left\|\varphi\right\|_{1,\alpha}.$

Conversely, assume that the function $\mathcal{F}_{W,\alpha}\varphi$ is an entire function of $y\in\mathbb{C}^{d+1}$ and satisfied (\ref{pw}), thus implies that the function 
$$\Phi_{l,\alpha}(\lambda)=
\lambda^{-1}\int_{S^d_+}\mathcal{F}_{W,\alpha}(\varphi)(\lambda,t)P_l(t)d\sigma_\alpha(t), \;\lambda\in\mathbb{C}$$
is an entire function of exponential type $R$, from which the converse follows from the Theorem \ref{tyt}.
\end{proof}
\begin{coro}\label{cc1}Let $\varphi\in \mathcal{S}_*(\mathbb{R}^{d+1}_+)$ be supported in $\left\{x;\;\left\|x\right\|<R\right\}$. Then  $\tau^\alpha_y\varphi$ is supported in $\left\{x;\;\left\|x\right\|<R+\left\|y\right\|\right\}$.
\end{coro}
\begin{proof}Let $\psi(x)=\tau_y^\alpha\varphi(x).$ Then $\mathcal{F}_{W,\alpha}\psi(z)=\Lambda^d_\alpha(y,z)\mathcal{F}_{W,\alpha}\varphi(z).$ Thus
$$\left|\mathcal{F}_{W,\alpha}\psi(z)\right|=\left|\Lambda^d_\alpha(y,z)\mathcal{F}_{W,\alpha}\varphi(z)\right|\leq c\;e^{(R+\left\|y\right\|)\left\|\Im z\right\|}.$$
i.e $\mathcal{F}_{W,\alpha}\psi(z)$ extends to $\mathbb{C}^{d+1}$ as an entire function of type $R+\left\|y\right\|$. Consequently, from the previous Theorem we conclude that  $\tau^\alpha_y\varphi$ is supported in $\left\{x;\;\left\|x\right\|<R+\left\|y\right\|\right\}$.
\end{proof}
\begin{coro}\label{cccc}Let $\varphi\in \mathcal{S}_*(\mathbb{R}^{d+1}_+)$. If $\varphi$ is supported in  $\left\{x;\;\left\|x\right\|<R\right\}$, then the following inequality
\begin{equation}
\left\|\tau^\alpha_x\varphi-\varphi\right\|_{\alpha, p}\leq \tilde{C}_{\alpha,d}\left\|x\right\|\left(R+\left\|x\right\|\right)^{\frac{d+2\alpha+2}{p}}
\end{equation}
holds for all $1\leq p\leq\infty$.
\end{coro}
\begin{proof}Inversion formula (\ref{inv}) and (\ref{tau}) yields that 
\begin{equation}\label{486}
\tau^\alpha_x\varphi(y)-\varphi(y)=C_{\alpha,d}\int_{\mathbb{R}^{d+1}_+} \big(\Lambda_\alpha^d(x,\xi)-1\big)\Lambda_\alpha^d(-y,\xi)\mathcal{F}_{W,\alpha}\varphi(\xi)d\mu_\alpha(\xi).
\end{equation}
Combining (\ref{klk}) and (\ref{486}) and using the  mean value theorem we get 
$$\big|\tau^\alpha_x\varphi(y)-\varphi(y)\big|\leq C_{\alpha,d}\left\|x\right\|\int_{\mathbb{R}^{d+1}_+}\left\|\xi\right\|\left|\mathcal{F}_{W,\alpha}\varphi(\xi)\right|d\mu_\alpha(\xi).$$ 
As $\varphi$ is supported in $\left\{y;\;\left\|y\right\|<R\right\}$ and $\tau^\alpha_x\varphi$ is supported in $\left\{y;\;\left\|y\right\|<R+\left\|x\right\|\right\},$ we can restrict the integration domain above to $\left\{y;\;\left\|y\right\|<R+\left\|x\right\|\right\}.$ A short calculation gives the desired  result.
\end{proof}

\begin{coro} \label{456789} If $\varphi\in L^1_\alpha(\mathbb{R}^{d+1}_+),$ then 
\begin{equation}
\lim_{t\longrightarrow0}\left\|\varphi*_WE^\alpha_t-\varphi\right\|_{\alpha,1}=0.
\end{equation}
\end{coro}
\begin{proof}From the fact $\tau^\alpha_x E^\alpha_t\geq0$ and  $\tau^\alpha_x\varphi(y)=\tau^\alpha_y\varphi(x)$ we get 
$$\left\|\varphi*_WE^\alpha_t\right\|_{\alpha,1}\leq \int_{\mathbb{R}^{d+1}_+}\int_{\mathbb{R}^{d+1}_+}\left|\varphi(y)\right|\tau^\alpha_xE^\alpha_t(y)d\mu_\alpha(y)d\mu_\alpha(x).$$
By Fubini's Theorem and the equality (\ref{;;}) we obtain 
\begin{equation}\label{ppp}
\left\|\varphi*_WE^\alpha_t\right\|_{\alpha,1}\leq\left\|\varphi\right\|_{\alpha,1}.
\end{equation}
For a given $\epsilon>0,$ we choose $\psi\in\mathcal{S}_*(\mathbb{R}^{d+1}_+)$ such that $\left\|\varphi-\psi\right\|_{\alpha,1}<\epsilon/3.$ Now, we can write $\varphi*_WE^\alpha_t-\varphi$ in the following form
$$\varphi*_WE^\alpha_t-\varphi=[\left(f-g\right)*_WE^\alpha_t]+[\psi-\varphi]+[\psi*_WE^\alpha_t-\psi].$$
So, the triangle inequality and (\ref{ppp}) leads to 
$$\left\|\varphi*_WE^\alpha_t-\varphi\right\|_{\alpha,1}\leq \frac{2\epsilon}{3}+\left\|\psi*_WE^\alpha_t-\psi\right\|_{\alpha,1}.$$
On the other hand, by using (\ref{hhh}) we have
\begin{equation*}
\begin{split}
\psi*_WE^\alpha_t(x)-\psi(x)&=\int_{\mathbb{R}^{d+1}_+}(\tau^\alpha_y\psi(-x)-\psi(x))E^\alpha_t(y)d\mu_\alpha(y)\\
&=\int_{\mathbb{R}^{d+1}_+}(\tau^\alpha_{-y}\psi(x)-\psi(x))E^\alpha_t(y)d\mu_\alpha(y)\\
&=\int_{\mathbb{R}^{d+1}_+}(\tau^\alpha_{y}\psi(x)-\psi(x))E^\alpha_t(y)d\mu_\alpha(y).
\end{split}
\end{equation*}
Thus implies then
$$\left\|\psi*_WE^\alpha_t\right\|_{\alpha,1}\leq\int_{\mathbb{R}^{d+1}_+}\left\|\tau^\alpha_y\psi-\psi\right\|_{\alpha,1}E^\alpha_t(y)d\mu_\alpha(y).$$
If $\psi$ is supported in  $\left\{x;\;\left\|x\right\|<R\right\}$, by Corollary \ref{cccc} we get 
\begin{equation*}
\begin{split}
\left\|\psi*_WE^\alpha_t\right\|_{\alpha,1}&\leq \tilde{C}_{\alpha,d}\int_{\mathbb{R}^{d+1}_+} \left\|y\right\|\left(R+\left\|y\right\|\right)^{d+2\alpha+2}E^\alpha_t(y)d\mu_\alpha(y)\\
&=\frac{2 \tilde{C}_{\alpha,d}}{\pi^{\frac{d}{2}}\Gamma(\alpha+1)(4t)^{\alpha+1+d/2}}\int_{\mathbb{R}^{d+1}_+} \left\|y\right\|\left(R+\left\|y\right\|\right)^{d+2\alpha+2}e^{-\left\|y\right\|^2/4t}d\mu_\alpha(y)\\
&=\frac{4 \tilde{C}_{\alpha,d}\sqrt{t}}{\pi^{\frac{d}{2}}\Gamma(\alpha+1)}\int_{\mathbb{R}^{d+1}_+} \left\|y\right\|\left(R+2\sqrt{t}\left\|y\right\|\right)^{d+2\alpha+2}e^{-\left\|y\right\|^2}d\mu_\alpha(y),
\end{split}
\end{equation*}
which can be made smaller than $\epsilon/3$ by choosing $\epsilon$ small. So, the proof of Corollary \ref{456789} is completes.
\end{proof}

 As a consequence of the Corollary \ref{456789} we reobtain the inversion formula (\ref{inv}).

\begin{coro} For $\varphi\in\mathcal{A}_\alpha(\mathbb{R}^{d+1}_+)$, then for almost $x\in\mathbb{R}^{d+1}_+$
\begin{equation}
\varphi(x)=C_{\alpha,d}\int_{\mathbb{R}^{d+1}_+}\mathcal{F}_{W,\alpha}\varphi(y)\Lambda_{\alpha}^d(x,-y)d\mu_\alpha(y)
\end{equation}
\end{coro}
\begin{proof}Let $\varphi\in\mathcal{S}(\mathbb{R}^{d+1}_+)$. Using (\ref{MM}), (\ref{MMM}) and (\ref{MMMM}) we have
\begin{equation*}
\begin{split}
\varphi*_WE^\alpha_t(x)&=\int_{\mathbb{R}^{d+1}_+}\tau^\alpha_x E^\alpha_t(-y)\varphi(y)d\mu_\alpha(y)\\
&=C_{\alpha,d}\int_{\mathbb{R}^{d+1}_+}\mathcal{F}_{\alpha,W}\big(\tau^\alpha_x E^d_\alpha(.)\big)(-y)\mathcal{F}_{\alpha,W}(\varphi)(y)d\mu_\alpha(y)\\
&=C_{\alpha,d}\int_{\mathbb{R}^{d+1}_+}e^{-t\left\|y\right\|^2}\Lambda_{\alpha}^d(x,-y)\mathcal{F}_{\alpha,W}(\varphi)(y)d\mu_\alpha(y).
\end{split}
\end{equation*}
This extends to $\varphi\in L^1_\alpha(\mathbb{R}^{d+1}_+),$ since the convolution operator extends to  $L^1_\alpha(\mathbb{R}^{d+1}_+),$ as a bounded operator by the inequality (\ref{ppp}). Letting $t\longrightarrow0^+,$ applying Corollary \ref{456789} to
the left-hand side and the dominant convergence theorem to the right-hand side, we see
that the inversion formula follows almost everywhere.
\end{proof}

\begin{coro}\label{cc}Suppose that $U\subseteq\mathbb{R}^{d+1}_+$ is open. Suppose, further that $x_1,x_2,...,x_N\in\mathbb{R}^{d+1}_+$ are pairwise distinct and $z_1,z_2,...,z_N\in\mathbb{C}.$ If $\sum_{k=1}^N z_k \Lambda^d_\alpha(x_k,\xi)=0$ for all $\xi\in U,$ then $z_k=0$ for all $k\in\left\{1,2,..., N\right\}.$
\end{coro}
\begin{proof}By successive analytic continuation in each coordinate we can derive that the assumption actually means that $\sum_{k=1}^N z_k \Lambda^d_\alpha(x_k,\xi)=0$ for all $\xi\in\mathbb{R}^{d+1}_+$. Let $f\in \mathcal{S}_*(\mathbb{R}^{d+1}_+)$ be a test function. Then 
$$0=\sum_{k=1}^N z_k \Lambda^d_\alpha(x_k,\xi)\mathcal{F}_{W,\alpha}f(\xi)=\mathcal{F}_{W,\alpha}\left(\sum_{k=1}^N z_k \tau^\alpha_{x_k}f(.)\right)(\xi),$$
for all $\xi\in\mathbb{R}^{d+1}_+$. Since $\tau^\alpha_{x_k}f\in\mathcal{S}_*(\mathbb{R}^{d+1}_+)$ and  The Weinstein transform $\mathcal{F}_{W,\alpha}$ is a topological isomorphism from $\mathcal{S}(\mathbb{R}^{d+1}_+)$ onto itself, we have 
\begin{equation}\label{xx}
\sum_{k=1}^N z_k \tau^\alpha_{x_k}f(\xi)=0,
\end{equation}
 for all $\xi\in\mathbb{R}^{d+1}_+$. Now take $f$ to be compactly supported  and support contained in the ball around zero with radius $\epsilon$ then for $\tau^\alpha_{x_k}f$ is compactly supported and support contained in the ball around zero with radius $\epsilon+\left\|x_k\right\|$ for all $k\in\left\{1,...,N\right\},$ by means of Corollary \ref{cc1}. Now, suppose that $\epsilon<\min_{j\neq k}\big|\left\|x_k\right\|-\left\|x_j\right\|\big|$ thus implies that 
$$\tau^\alpha_{x_k}f(x_j)=0,\textrm{for\;all\;} j\neq k,\;\;\textrm{and}\;\;\tau^\alpha_{x_k}f(x_k)\neq0.$$
So, by again using (\ref{xx}) we obtain that $z_k\tau^\alpha_{x_k}f(x_k)=0$ for all $k\in\in\left\{1,...,N\right\}$ and consequently $z_k=0.$
\end{proof}

\end{document}